\documentclass[a4paper, 10pt]{article}
\usepackage{graphicx}
\usepackage{alltt}
\usepackage{amsmath}
\usepackage[hidelinks, pdftex]{hyperref}

\usepackage{tabulary}
\usepackage{color}
\usepackage{amssymb}
\usepackage{amsmath,amssymb}
\usepackage{amsthm}
\usepackage{bbm}
\newtheorem{thm}{Theorem}

\newtheorem{lem}[thm]{Lemma}

\usepackage{hyperref}
\newcommand\dist{\buildrel d \over =}

\begin{document}

\begin{center}
\LARGE
\textbf{Ruin probability for renewal risk models with neutral net profit condition}\\[6pt]
\medskip
\small
\textbf {Andrius Grigutis, Arvydas Karbonskis, Jonas Šiaulys}\\[6pt]
Institute of Mathematics, Vilnius University,\\ Naugarduko 24, Vilnius, Lithuania,  LT-03225 \\
\medskip
andrius.grigutis@mif.vu.lt, arvydas.karbonskis@mif.stud.vu.lt,\\ jonas.siaulys@mif.vu.lt\\[6pt]
\end{center}

\begin{abstract}
In ruin theory, the net profit condition intuitively means that the incurred random claims on average do not occur more often than premiums are gained. The breach of the net profit condition causes guaranteed ruin in few but simple cases when both the claims' inter-occurrence time and random claims are degenerate. In this work, we give a simplified argumentation for the unavoidable ruin when the incurred claims on average occur equally as the premiums are gained. We study 
the discrete-time risk model with $N\in\mathbb{N}$ periodically occurring independent distributions, the classical risk model, also known as Cramér–Lundberg risk process, and the more general E. Sparre Andersen model. 

 \vskip 2mm

\textbf{Keywords:} net profit condition, ruin probability, discrete-time risk model, classical risk model, E.S. Andersen risk model, random walk.

 \vskip 2mm

 \textbf{MSC 2020:} 60G50, 60J80, 91G05. 
\end{abstract}

\nocite{2009ProcDETAp}

\section{Introduction}\label{ii}

In 1957, during the 15th International Congress of Actuaries, E. Sparre Andersen \cite{AES} proposed to use a \textit{renewal risk model} to describe the behavior of the insurer's surplus. According to Andersen's proposed model, the insurer's surplus process $W$ admits the following representation
\begin{align}\label{01}
W(t)=u+ct-\sum_{i=1}^{\Theta(t)}X_i,\ t\geqslant 0,
\end{align}
where:\\
\indent $\bullet$ $u\geqslant 0$ denotes the initial insurer's surplus, $W(0)=u$;\\
\indent $\bullet$ $c>0$ denotes the premium rate per unit of time;\\
\indent $\bullet$ the cost of claims $ X_1,\,X_2,\,\ldots$ are independent copies of  a non-negative random variable $X$;\\
\indent $\bullet$ the inter-occurrence times $\theta_1,\,\theta_2,\,\ldots $ between claims
 is another sequence of independent
copies of a non-negative random variable $\theta$ which is not degenerate at zero, i. e. $\mathbb{P}(\theta=0)<1$;\\
\indent
$\bullet$ the sequences $\{X_1,\,X_2,\,\ldots\}$ and
$\{\theta_1,\,\theta_2,\,\ldots\}$ are mutually independent;
\\
\indent $\bullet$ $\Theta(t)=\#\{n\geqslant 1: T_n\in[0,t]\}$ is the renewal process generated by the random variable $\theta$, where
$T_n=\theta_1+\theta_2+\ldots +\theta_n$.

    The main critical characteristics of the defined renewal risk model \eqref{01} are the \textit{time of ruin }
\begin{align*}
\tau_u=\begin{cases}&\inf\{t\geqslant 0: W(t)<0\},\\
& \infty, \mbox{\ if\ }\ W(t)\geqslant 0 \ \mbox{\ for all\ }\ t\geqslant 0 \end{cases}
\end{align*}
and the \textit{ultimate time ruin probability} (or just the \textit{ruin probability})
\begin{align*}
\psi(u)=\mathbb{P}(\tau_u<\infty).
\end{align*}
The model \eqref{01} and the definition of $\psi(u)$ imply that for all $u\geqslant 0$
\begin{align}\label{2}
\psi(u)&=\mathbb{P}\left(\bigcup_{t\geqslant 0}\left\{W(t)<0\right\}\right)\nonumber\\
&=\mathbb{P}\left(\inf_{n\geqslant1}\left\{u+c T_n-\sum_{i=1}^{n}X_n\right\}<0\right)\nonumber\\
&=\mathbb{P}\left(\sup_{n\geqslant1}\sum_{k=1}^n(X_k-c \theta_k)>u\right).
\end{align}

Thus, the ultimate time ruin probability $\psi(u)$ is nothing but the tail of the distribution function of the random variable $\sup_{n\geqslant1}\sum_{k=1}^n(X_k-c \theta_k)$. In ruin theory the difference $\mathbb{E}X-c\,\mathbb{E}\theta$ describes the so-called {\it net profit condition}. It is well known that $\psi(u)=1$ for any $u\geqslant 0$ if $\mathbb{E}X-c\,\mathbb{E}\theta>0$ where this fact is easily implied by the strong law of large numbers, see \cite[Prop. 7.2.3]{Resnic}. Also, $\psi(u)=1$ for any $u\geqslant 0$ if $\mathbb{E}X-c\,\mathbb{E}\theta=0$ (see \cite[pp. 559-564]{Resnic}), except in some simple cases when both random variables $X$ and $\theta$ are degenerate. We call the net profit condition to be {\it neutral} if $\mathbb{E}X-c\,\mathbb{E}\theta=0$ and say that it holds if $\mathbb{E}X-c\,\mathbb{E}\theta<0$. In general, the latter fact that 
\begin{align}\label{zero_net}
\mathbb{E}X-c\,\mathbb{E}\theta=0 \quad \Rightarrow \quad \psi(u)=1
\end{align}
for all $u\geqslant 0$,  can be deduced from some deep study of random walk, see for example \cite{Feller}, \cite{Resnic}, \cite{Spitzer}. Therefore, the mathematical curiosity drives us to derive \eqref{zero_net} by using simpler arguments. 

In \cite{DS}, authors basically use Silverman–Toeplitz theorem to prove \eqref{zero_net} for the discrete-time and classical risk models. The proofs presented for both models are significantly simpler than those presented in \cite{Feller}, \cite{Resnic}, \cite{Spitzer}. In this article, we show that the implication \eqref{zero_net} can be simplified even further, however in some instances using the Pollaczek-Khinchine formula.  The desired simplification of the proof can be achieved by defining the random vector $(X^*,\, X)$, where $X^*$ is the new random variable which is arbitrarily close to $X$ and $\mathbb{P}(X^*\leqslant X)=1$.\footnote{Originaly the idea was raised by the fourth-year student of Faculty of Mathematics and Informatics Justas Klimavičius in 2017.} This way is similar to the probabilistic proof of the Turan's theorem given in \cite{AZ}\footnote[2]{We thank Professor Eugenijus Manstavičius for pointing to this fact.}. For the defined random variable $X^*$ we make the net profit condition satisfied $\mathbb{E}X^*-c\mathbb{E}\theta<0$ and show that the known algorithms of the ruin probability calculation under the net profit condition imply $\psi(u)=1$ for all $u\geqslant 0$ as $X^*$ approaches to $X$.

    In Section \ref{22} we derive \eqref{zero_net} for the more general discrete-time risk model when $\theta \equiv 1$, $c\in\mathbb{N}$ and non-negative independent integer-valued random variables $X_i\dist X_{i+N}$ for all $i\in\mathbb{N}$ and some fixed natural $N$, i.e. we allow the random variables $X_1,\,X_2,\,\ldots$ in model \eqref{01} to be independent but not necessarily identically distributed. Obviously, if $N=1$ then we get that r.v.s $X_1,\,X_2,\,\ldots$ are identically distributed. In Section \ref{sec:Classical}, we derive \eqref{zero_net} for the classical risk model when $\Theta(t)$ in \eqref{01} is assumed to be a Poison process with intensity $\lambda>0$. Recall that in this case
$$
\mathbb{P}(\Theta(t+s)-\Theta(s)=n)=e^{-\lambda t}\frac{(\lambda t)^n}{n!}
$$
for all $n\in\mathbb{N}$ and  $t,\,s>0$. In the last Section \ref{sec:Andersen}, we consider the most general E.S. Andersen's model \eqref{01} in terms of proving \eqref{zero_net} by the known facts of ruin probability calculation under the net profit condition. More precisely, we reformulate and give different proofs than the existing ones to the following three theorems.

\begin{thm}\label{thm:discrete}
Suppose the insurer's surplus process $W(t)$  varies according to the discrete-time risk model \eqref{discrete_time} with $N$ periodically occurring independent discrete and integer-valued non-negative r.v.s $X_i\mathop{=}\limits^{d}X_{i+N}$ and $\theta\equiv 1$. Let $S_N=X_1+X_2+\ldots+X_N$. If the net profit condition is neutral $cN-\mathbb{E} S_N=0$ and $\mathbb{P}(S_N=cN)<1$, the ultimate time ruin probability $\psi(u)=1$ for all $u\in\mathbb{N}\cup\{0\}$.
\end{thm}

\begin{thm}\label{thm:classical}
Let $W(t),\,t\geqslant0$ be a surplus process of the classical risk model generated by a random claim amount $X$, an exponentially distributed inter-occurrence time $\theta$ with mean $\mathbb{E}\theta=1/\lambda,\,\lambda>0$, and a constant premium rate $c>0$. If the net profit condition is neutral $\lambda\mathbb{E}X=c$, then $\psi(u)=1$ for all $u\geqslant 0$.
\end{thm}

\begin{thm}\label{thm:Andersen}
Let $W(t),\,t\geqslant0$ be a surplus process of E. Sparre Andersen model generated by a random claim amount $X$, inter-occurrence time $\theta$, and a constant premium rate $c>0$. If the net profit condition is neutral $\mathbb{E}X/\mathbb{E}\theta=c$ and $\mathbb{P}(X=c\theta)<1$, then $\psi(u)=1$ for all $u\geqslant 0$.
\end{thm}

\section{One auxiliary statement}
Proving Theorems \ref{thm:classical} and \ref{thm:Andersen} we use the Pollaczek–Khinchine formula. This raises the need for the following statement.

\begin{lem}\label{lem:concentration}
Let $\eta_1,\,\eta_2,\,\ldots$ be independent identically distributed non-negative random variables which are not degenerate at zero. Then
$$
\sum_{n=1}^{\infty}\mathbb{P}(\eta_1+\ldots+\eta_n\leqslant x)<\infty
$$
for any $x\geqslant0$.
\end{lem}
\begin{proof}
Let $t$ be some small positive number and say that the non-negative random variables $\eta_1,\,\eta_2,\,\ldots$ are independent copies of $\eta$. Then, rearranging and using Markov's inequality, we obtain 
\begin{align*}
\sum_{n=1}^{\infty}\mathbb{P}(\eta_1+\ldots+\eta_n\leqslant x)
=\sum_{n=1}^{\infty}\mathbb{P}\left(e^{-t(\eta_1+\ldots+\eta_n)}\geqslant e^{-tx}\right)
\leqslant e^{tx}\sum_{n=1}^{\infty}\left(\mathbb{E}e^{-t\eta}\right)^n<\infty,
\end{align*}
since $\mathbb{E}e^{-t\eta}<1$ under the considered conditions.
\end{proof}
Of course, the upper bound of the sum $\sum_{n=1}^{\infty}\mathbb{P}(\eta_1+\ldots+\eta_n\leqslant x)$ can be improved compared to the given one; see for instance \cite[Proof of lem. 8]{RJ} and other literature on concentration inequalities.

\section{Discrete-time risk model}\label{22}

Let us consider the model \eqref{01}. Suppose $c\in\mathbb{N}$, $\theta\equiv 1$, the independent random variables $X_1,\,X_2,\,\ldots$ are non-negative integer-valued and follow the N-seasonal pattern, i.e. $X_i\dist X_{i+N}$ for all $i\in\mathbb{N}$ and some fixed $N\in \mathbb{N}$. If these requirements are satisfied, then the general E. S. Andersen's renewal risk model \eqref{01} becomes the discrete-time risk model
\begin{align}\label{discrete_time}
W(t)=u+ct-\sum_{i=1}^{\lfloor t\rfloor}X_i,\, t\geqslant 0,
\end{align}
where symbol $\lfloor\cdot \rfloor$ denotes the floor function.
Then, there is sufficient to consider \eqref{discrete_time} (in terms of $W(t)<0$ for at least one $t\geqslant0$) when $u\in\{0,\,1,\,2,\,\ldots\}=:\mathbb{N}_0$ and $ t\in\mathbb{N}$ only. Then, the ruin time and the ultimate time ruin probability have the following standard expressions
\begin{align}
&\tau_u=\begin{cases}&\min\{t\in\mathbb{N}: W(t)<0\},\\
& \infty, \mbox{\ if\ }\ W(t)\geqslant 0 \ \mbox{\ for all\ }\ t\in\mathbb{N}, \end{cases}\nonumber\\
&\psi(u)=\mathbb{P}\left(\tau_{u}<\infty\right)=\mathbb{P}\left(\sup_{k\geqslant
1}\sum_{i=1}^k(X_i-c)>u\right),\,u\in\mathbb{N}_0.\label{ruin_p_disc}
\end{align}

If we denote $\varphi=1-\psi$ the ultimate time survival probability, then, according to \eqref{ruin_p_disc},
\begin{align}\label{surv_p}
\varphi(u)=\mathbb{P}\left(\sup_{k\geqslant
1}\sum_{i=1}^k(X_i-c)\leqslant u\right),\,
u\in\mathbb{N}_0.
\end{align}

In \cite{SA} and various other papers, the survival probability is studied according to a slightly different definition than \eqref{surv_p}, i.e.
\begin{align}\label{surv_p_<}
\hat{\varphi}(u)=\mathbb{P}\left(\sup_{k\geqslant
1}\sum_{i=1}^k(X_i-c)< u\right).
\end{align}
It is easy to see that 
\begin{align}\label{eq_def}
\varphi(u)=\hat{\varphi}(u+1)
\end{align}
for all $u\in\mathbb{N}_0$. We now prove Theorem \ref{thm:discrete}.

  \begin{proof}[Proof of Theorem \ref{thm:discrete}]   
    We first demonstrate the proof for the most simplistic version of the homogeneous discrete-time risk model \eqref{discrete_time} when $c=1$ and $N=1$. Let $h_k=\mathbb{P}(X=k)$, $k\in\mathbb{N}_0$ and observe that conditions $\mathbb{E}X=1$ and $\mathbb{P}(X=1)=h_1<1$ imply $h_l>0$ for some $l\geqslant 2$. Indeed,
$$
\mathbb{E}X=h_1+2h_2+3h_3+\ldots=1
$$
and $h_1<1$ means that at least one probability out of $h_2,\,h_3,\,\ldots$ is positive. In addition, conditions $h_1<1$ and $\mathbb{E}X=1$ imply $h_0>0$. Indeed, if $h_0=0$, then $h_1+h_2+h_3+\ldots =1$ and
$$
1=\mathbb{E}X=h_1+2h_2+3h_3+\ldots>h_1+h_2+h_3+\ldots =1
$$
leads to the contradiction.

    Let us choose $l\geqslant 2$ such that $h_l=\mathbb{P}(X=l)>0$ and define the distribution of an integer-valued random vector $(X^*,\,X)$ by the following equalities:
    \begin{eqnarray*}
    &&\mathbb{P}(X^*=k,\,X=k)=h_k,\, k\in\mathbb{N}_0, k\neq l,\\
    &&\mathbb{P}(X^*=l,\,X=l)=h_l-\frac{\varepsilon}{l},\\
    &&\mathbb{P}(X^*=0,\,X=l)=\frac{\varepsilon}{l},\\
    &&\mathbb{P}(X^*=k,\,X=m)=0,\, \{k,\,m\}\in\mathbb{N}_0^2,\, \{k,\,m\}\neq \{0,\, l\},\,k\neq m,
        \end{eqnarray*}
        where $\varepsilon\in(0,\,l h_l)$ is arbitrary small.

Visually,  vector's $(X^*,\,X)$ distribution is the following

\begin{center}
\begin{tabular}{|c|c|c|c|c|c|c|c|c||c|}
\hline
$X^*\backslash X$&$0$&$1$&$2$&$\ldots$&$l-1$&$l$&$l+1$&$\ldots$&$\Sigma$\\
\hline
$0$&$h_0$&$0$&$0$&$\ldots$&$0$&$\varepsilon/l$&$0$&$\ldots$&$h_0+\varepsilon/l$\\
\hline
$1$&$0$&$h_1$&$0$&$\ldots$&$0$&$0$&0&$\ldots$&$h_1$\\
\hline
$2$&$0$&$0$&$h_2$&$\ldots$&$0$&$0$&$0$&$\ldots$&$h_2$\\
\hline
$\vdots$&$\vdots$&$\vdots$&$\vdots$&$\ddots$&$\vdots$&$\vdots$&$\vdots$&$\ddots$&$\vdots$\\
\hline
$l-1$&$0$&$0$&$0$&$\ldots$&$h_{l-1}$&$0$&$0$&$\ldots$&$h_{l-1}$\\
\hline
$l$&$0$&$0$&$0$&$\ldots$&$0$&$h_l-\varepsilon/l$&$0$&$\ldots$&$h_{l}-\varepsilon/l$\\
\hline
$l+1$&$0$&$0$&$0$&$\ldots$&$0$&$0$&$h_{l+1}$&$\ldots$&$h_{l+1}$\\
\hline
$\vdots$&$\vdots$&$\vdots$&$\vdots$&$\ddots$&$\vdots$&$\vdots$&$\vdots$&$\ddots$&$\vdots$\\
\hline
\hline
$\Sigma$&$h_0$&$h_1$&$h_2$&$\ldots$&$h_{l-1}$&$h_l$&$h_{l+1}$&$\ldots$&$1$\\
\hline
\end{tabular}
\end{center}

    It is easy to see, that $\mathbb{E}X^*=1 -\varepsilon<1$, and
\begin{align*}
\mathbb{P}(X^*\leqslant X)&=\sum_{k=0}^{\infty}\sum_{m=k}^{\infty}\mathbb{P}(X^*=k,\,X=m)\\
&=\mathbb{P}(X^*=0,\,X=l)+\sum_{k=0}^{\infty}\mathbb{P}(X^*=k,\,X=k)=1.
\end{align*}

    Let $(X^*_j,\,X_j)$, $j\in\mathbb{N}$, be  independent copies of random vector $(X^*,\,X)$. We have that $\mathbb{P}(X^*_j\leqslant X_j)=1$ for each $j\in\mathbb{N}$. Therefore,
\begin{align*}
\mathbb{P}(X_1^*+X_2^*\leqslant X_1+X_2)&=\sum_{k=0}^{\infty}\sum_{l=0}^{\infty}\mathbb{P}(X_1^*+k\leqslant X_1+l)\mathbb{P}(X_2^*=k,\,X_2=l)\\
&=\sum_{k=0,\ k\neq l}^{\infty}\mathbb{P}(X_1^*+k\leqslant X_1+k)h_k\\
&\hspace{2mm}+\ \mathbb{P}(X_1^*+l\leqslant X_1+l)\left(h_l -\frac{\varepsilon}{l}\right)\\
&\hspace{2mm}+\ \mathbb{P}(X_1^*\leqslant X_1+l)\frac{\varepsilon}{l}=1,
\end{align*}
due to $\mathbb{P}(X_1^*\leqslant X_1)=1$.

We now use the mathematical induction to show 
\begin{equation}\label{P}
\mathbb{P}\left(\sum_{k=1}^nX_k^*\leqslant\sum_{k=1}^nX_k\right)=1,\, n\in \mathbb{N}.
\end{equation}
Indeed, if $\mathbb{P}\left(\sum_{k=1}^nX_k^*\leqslant\sum_{k=1}^nX_k\right)=1$ up to some natural $n$, then we get that
\begin{align*}
\mathbb{P}\left(\sum_{k=1}^{n+1}X_k^*\leqslant\sum_{k=1}^{n+1}X_k\right)&=\sum_{k=0,\ k\neq l}^{\infty}\mathbb{P}\left(\sum_{k=1}^{n}X_k^*\leqslant\sum_{k=1}^{n}X_k\right)h_k\\
&\hspace{2mm}+\ \mathbb{P}\left(\sum_{k=1}^{n}X_k^*\leqslant\sum_{k=1}^{n}X_k\right)\left(h_l-\frac{\varepsilon}{l}\right)\\
&\hspace{2mm}+\ \mathbb{P}\left(\sum_{k=1}^{n}X_k^*\leqslant\sum_{k=1}^{n}X_k+l\right)\frac{\varepsilon}{l}=1.
\end{align*}

For $u\in\mathbb{N}_0$, the equality (\ref{P}) implies that
\begin{align*}
\psi(u)&=\mathbb{P}\left(\sup_{n\geqslant 1}\left\{\sum_{k=1}^nX_k-n\right\}>u\right)
=\mathbb{P}\left(\bigcup_{n=1}^\infty\left\{ \sum_{k=1}^n X_k>n+u\right\}\right)\\
&\geqslant\mathbb{P}\left(\bigcup_{n=1}^\infty\left\{ \sum_{k=1}^n X_k^*>n+u\right\}\right)
= \mathbb{P}\left(\sup_{n\geqslant 1}\left\{\sum_{k=1}^nX^*_k-n\right\}>u\right)\\&=:\psi_{\varepsilon}^*(u),
\end{align*}
or, equivalently, 
\begin{align}\label{ineq}
\varphi(u)\leqslant \varphi_{\varepsilon}^*(u),
\end{align}
for all $u\in\mathbb{N}_0$, where $\varphi=1-\psi$ and $\varphi_{\varepsilon}^*=1-\psi^*_{\varepsilon}$ are the model's survival probabilities.
 
    Let $s\in\mathbb{C}$ and $h^*_k=\mathbb{P}(X^*=k)$, $k\in\mathbb{N}_0$. Since $\mathbb{E}X^*=1-\varepsilon<1$, Corollary 3.2 of \cite{GR} implies that  the  probability generating function of the survival probability  $\varphi^*$ satisfies the following equation
\begin{align}\label{gen_phi}
\varphi^*(0)+\varphi^*(1)s+\varphi^*(2)s^2+\ldots=\frac{1-\mathbb{E}X^*}{G_{X^*}(s)-s}=\frac{\varepsilon}{G_{X^*}(s)-s},\,|s|<1,
\end{align}
where $G_{X^*}(s)$ is the probability generating function  of r.v. $X^*$, i.e.
$$
G_{X^*}(s)=h^*_0+h^*_1s+h^*_2s^2+\ldots,\,|s|\leqslant 1.
$$
Inequality \eqref{ineq} and  equation  \eqref{gen_phi} imply that
\begin{align*}
0\leqslant \varphi(0)\leqslant\frac{\varepsilon}{h^*_0}=\frac{\varepsilon}{h_0+\varepsilon/l},\ \ 
0\leqslant \varphi(1)\leqslant\frac{\varepsilon(1-h_1)}{(h_0+\varepsilon/l)^2},
\end{align*}
and, in general,  
\begin{align*}
0\leqslant\varphi(n)\leqslant\varepsilon \cdot \frac{1}{n!}\lim_{s\to 0}\frac{d^n}{ds^n}\left(\frac{1}{G_{X^*}(s)-s}\right)
\end{align*}
for all $n\in\mathbb{N}_0$. Since  $\varepsilon$ can be arbitrarily small,  we conclude that $\varphi(u)=0$ or, equivalently, $\psi(u)=1$ for all $u\in\mathbb{N}_0$. 

It is worth mentioning that, having $\varphi(0)=0$, the equality $\varphi(u)=0$ for all $u\in\mathbb{N}$ can be  concluded from the following recurrence formula  (see, for instance, \cite[Section 6]{Dickson}, \cite{Dickson_Walters}, \cite{Shiu}, \cite{Shiu1})
\begin{align}\label{eq:recurr}
\varphi(u)=\frac{1}{h_0}\left(\varphi(u-1)-\sum_{k=1}^{u}\varphi(u-k)h_k\right),\,u\in\mathbb{N}.
\end{align}
 Indeed, the recurrence \eqref{eq:recurr} yields $\varphi(u),\,u\in\mathbb{N}_0$ being the multiple of $\varphi(0)=1-\mathbb{E}X$. More precisely,
$$
\varphi(u)=\alpha_u \varphi(0),
$$
with 
$$
\alpha_0=1,\,\alpha_u=\frac{1}{h_0}\left(\alpha_{u-1}-\sum_{k=1}^{u}\alpha_{u-k} h_k\right),\,u\in\mathbb{N}.
$$
The latter expression  can be verified by mathematical induction. So, the particular case with  $c=1$ and $N=1$ in Theorem \ref{thm:discrete} is proved. 

\medskip

The general case when $c\in\mathbb{N}$ and $N\in \mathbb{N}$ in the discrete-time risk model \eqref{discrete_time} can be considered by the same means. Let us explain how. 

Let us suppose  the model \eqref{discrete_time} is generated by $X_1,\,X_2,\,\ldots,\,X_N$ periodically occurring independent non-negative and integer-valued random variables, i.e.\\ $X_{i+N}\dist X_i$ for all $i\in\mathbb{N}$ and some fixed $N\in\mathbb{N}$. In such a case  we can choose any random variable from  $\{X_1,\,X_2,\,\ldots,\,X_N\}$ and define the random vector $(X^*_j,\,X_j)$ such that $\mathbb{P}(X^*_j\leqslant X_j)=1$ where $j\in\{1,\,2,\,\ldots,\,N\}$ is some fixed number. Obviously, the random vector $(X^*_j,\,X_j)$ must be defined in a similar way as vector  $(X^*,\,X)$ before, where both random variables $X^*_j$ and $X_j$ attain the same values and the probability of some smaller value of $X^*_j$ gets enlarged by some arbitrarily small value and the probability of some larger value of $X^*_j$ gets reduced by the same size. Note that conditions  $\mathbb{P}(X_j\geqslant c)=1$ and $\mathbb{P}(S_N=cN)<1$ imply the estimate  $cN-\mathbb{E}S_N<0$ which is not the case under consideration. Hence, always there exists at least one value in the set $\{0,\,1,\,\ldots,\,c-1\}$ for r.v. $X_j$ which we can choose to enlarge its probability defining $X^*_j$.  Then we achieve 
$$
\varepsilon:=cN-\mathbb{E}S^*_N>cN-\mathbb{E}S_N=0,
$$
where $S^*_N=X_1+\ldots+X^*_j+\ldots+X_N$. By the same arguments  as deriving inequality \eqref{ineq}, we get that 
$0\leqslant\varphi(0)\leqslant\varphi_\varepsilon^*(0)$, where $\varphi_\varepsilon^*(0)$ is the ultimate time survival probability at $u=0$ for the model in which r.v. $X^*_j$ replaces r.v. $X_j$ for some $j\in\{1,2,\ldots,N\}$. According to \cite[Thm. 4]{SA} we obtain  
\begin{align}\label{phi_0_N}
\varphi_\varepsilon^*(0)=\frac{m^{*(1)}_0}{\mathbb{P}(S^*_N=0)}
\end{align}
if  $\mathbb{P}(S^*_N=0)>0$,
where $m^{*(1)}_0$ is the first component of the solution of the following system of linear equations
\begin{align}\label{main_syst}
M_{c N\times c N}
\times
\begin{pmatrix}
m^{*(1)}_0\\
m^{*(1)}_1\\
\ldots\\
m^{*(1)}_{c-1}\\
m^{*(2)}_0\\
m^{*(2)}_1\\
\ldots\\
m^{*(2)}_{c-1}\\
\ldots\\
m^{*(N)}_0\\
m^{*(N)}_1\\
\ldots\\
m^{*(N)}_{c-1}
\end{pmatrix}_{c N\times1}\hspace{-5mm}
=
\begin{pmatrix}
0\\
\vdots\\
0\\
c N-\mathbb{E}S^*_N
\end{pmatrix}_{c N\times1}\hspace{-0.5cm}
=
\begin{pmatrix}
0\\
\vdots\\
0\\
\varepsilon
\end{pmatrix}_{c N\times1}\hspace{-0.5cm},
\end{align}
where $M_{c N\times c N}$ is a certain matrix with elements related to the roots of equation $G_{S^*_N}(s)=s^{cN},\,|s|\leqslant 1$ (see \cite[Sec. 3]{SA}). Letting $\varepsilon\to0^+$ 
we derive from the system \eqref{main_syst} that $\varphi_\varepsilon^*(0)\to 0$ because of \eqref{phi_0_N}. Consequently $\varphi(0)=0$ due to the estimate $0\leqslant \varphi(0)\leqslant \varphi_\varepsilon^*(0)$ provided for an arbitrary $\varepsilon >0$. It should be noted that the requirement  $\mathbb{P}(S^*_N=0)>0$ for equality  \eqref{phi_0_N} does not reduce generality, because $\mathbb{P}(S^*_N=0)$ can be replaced by the probability of the smallest value of $S^*_N$ if $\mathbb{P}(S^*_N=0)=0$, see the comments in \cite[Sec. 4]{SA}. In addition, the non-singularity of the matrix $M_{c N\times c N}$ in \eqref{main_syst} is not known in general, see \cite[Sec. 4]{SA} and \cite{GR}, also \cite{GJ}. On the other hand, if $c\in\mathbb{N}$, $N=1$ and the roots of $G_{X^*}=s^c$ are simple, the solution of \eqref{main_syst} admits the closed-form expression and, obviously, $m^{*(1)}_0$ is the multiple of $c-\mathbb{E}X^*=\varepsilon$, see \cite{GR}. In cases when the non-singularity of the matrix $M_{c N\times c N}$ in \eqref{main_syst} remains questionable, we can refer to \cite[Thm. 3]{SA}, for the different proof that $\varphi(0)=0$ if the net profit condition is neutral $\mathbb{E}S_N=cN$ and $\mathbb{P}(S_N=cN)<1$.

Having $\varphi(0)=0$, the remaining values $\varphi(u)=0,\,u\in\mathbb{N}$ can be  obtained  by the recurrence relation
$$
\varphi(u)
=\hspace{-0.5cm}\sum_{\substack{i_1\leqslant u+c\\i_1+i_2\leqslant u+2c\\
\ldots \vspace{1mm} \\
i_1+i_2+\ldots+i_N\leqslant u+c N}}\hspace{-5mm}\mathbb{P}(X_1=i_1)\mathbb{P}(X_2=i_2)\cdots\mathbb{P}(X_N=i_N)\,
\varphi\left(u+c N-\sum_{j=1}^Ni_j\right),
$$
see \cite[eq. (5)]{SA} or by the following expression of survival probability generating function (see \cite[Thm. 2]{SA})
\begin{align*}
\varphi_\varepsilon^*(0)+\varphi_\varepsilon^*(1)s+\varphi_\varepsilon^*(2)s^2+\ldots=\frac{\mathbbm{u}^T\mathbbm{v}}{G_{S^*_N}(s)-s^{c N}},
\end{align*}
where, having in mind that some $X_j$ from  $\{X_1,\,\ldots,\,X_n\}$ is replaced by $X^*_j$,
\begin{align*}
\mathbbm{u}=\begin{pmatrix}
s^{c(N-1)}\\s^{c(N-2)}G_{S^*_1}(s)\\s^{c(N-3)}G_{S^*_2}(s)\\ \vdots\\s^c G_{S^*_{N-2}}(s)\\G_{S^*_{N-1}}(s)
\end{pmatrix},\,
\mathbbm{v}=
\begin{pmatrix}
&\sum\limits_{i=0}^{c-1}m_k^{*(2)}\sum\limits_{k=i}^{c-1}s^k F_{X_1}(k-i)\\
&\sum\limits_{i=0}^{c-1}m_i^{*(3)}\sum\limits_{k=i}^{c-1}s^k F_{X_2}(k-i)
\\
&\vdots\\
&\sum_{i=0}^{c-1}m_i^{*(j+1)}\sum\limits_{k=i}^{c-1}s^k F_{X_j^*}(k-i)\\
&\vdots\\
&\sum_{i=0}^{c-1}m_i^{*(N)}\sum\limits_{k=i}^{c-1}s^k F_{X_{N-1}}(k-i)\\
&\sum_{i=0}^{c-1}m_i^{*(1)}\sum\limits_{k=i}^{c-1}s^k F_{X_N}(k-i)\\
\end{pmatrix},
\end{align*}
and
$$
G_{S_l^*}(s),\,|s|\leqslant 1,\,l\in\{1,\,2,\,\ldots,\,N-1\}
$$
is the probability generating function of random variable
\begin{equation*}
 S^*_l=\begin{cases}
 X_1+\ldots+X_l & {\rm if}\  l<j,\\
 X_1+\ldots+X_{j-1}+X_j^*+X_{j+1}+\ldots +X_l & {\rm  if}\  l\geqslant j, 
\end{cases}   
\end{equation*}
$F_{X_l}$  is the distribution function of $X_l$ and the collection $$\left\{m^{*(1)}_0,\,m^{*(1)}_1,
 \ldots, m^{*(1)}_{c-1},\, m^{*(2)}_0,\,m^{*(2)}_1, \ldots,\,m^{*(2)}_{c-1}, \ldots, m^{*(N)}_0,\,
m^{*(N)}_1,\,\ldots,\,m^{*(N)}_{c-1}\right\}$$ satisfies the system \eqref{main_syst} being the multiple of $cN-\mathbb{E}S^*_N$.
\end{proof}

\section{Classical risk model}\label{sec:Classical} 
In this section we prove Theorem \ref{thm:classical}.

\begin{proof}[Proof of Theorem \ref{thm:classical}]
Since the random variable $X$ in model \eqref{01} is non-negative and $X\equiv0$ is out of options for the considered stochastic process, then $\mathbb{E}X > 0$ and there exists $a>0$ such that $\mathbb{P}(X>a)>0$. Similarly as proving Theorem \ref{thm:discrete}, we now define the pair of dependent random variables $(X^*,\, X)$ where $X^*$ for any $\varepsilon\in(0,a)$ is
$$
X^*=
\begin{cases} 
X - \varepsilon &{\rm if}\  X > a, \\
X & {\rm if }\  X \leqslant a. 
 \end{cases}
$$
For this new r.v. 
\begin{align*}
\mathbb{E}X^* &=\mathbb{E}X-\varepsilon\mathbb{P}(X>a)<\mathbb{E}X,\\
\mathbb{P}(X^* \leqslant X)&=\mathbb{P}(X^* \leqslant X, X>a)+\mathbb{P}(X^* \leqslant X, X\leqslant a)=1.
\end{align*} 
Let $(X^*_j,\, X_j),\, j=1,\,2,\,\ldots$ be independent copies of $(X^*,\, X)$. Then we have:
\begin{align*}&\mathbb{P}\left(X^*_j \leqslant X_j\right)=1,\text{ for all } j\in \mathbb{N},\\
&\mathbb{P}\left(\sum_{j=1}^{n}X^*_j \leqslant \sum_{j=1}^{n}X_j\right)=1,\text{ for all } n\in \mathbb{N},\\
&\mathbb{P}\left(\sum_{j=1}^{n}(X^*_j-c\theta_j) \leqslant \sum_{j=1}^{n}(X_j-c\theta_j)\right)=1,\text{ for all } n\in \mathbb{N},\\
&\mathbb{P}\left(\sup_{n\geqslant1}\sum_{j=1}^{n}(X^*_j-c\theta_j)\leqslant\sup_{n\geqslant1}\sum_{j=1}^{n}(X_j-c\theta_j)\right)=1,
\end{align*}
and, by similar arguments as in (\ref{ineq}),
$\psi_\varepsilon^*(u)\leqslant \psi(u)\leqslant1$ for all $u\geqslant0$. Conditions  $$\mathbb{E}X^*=\mathbb{E}X-\varepsilon\mathbb{P}(X>a),\ \lambda \mathbb{E}X/c=1$$ and  well-known formula for  $\psi_\varepsilon^*(0)$ (see, for example, \cite{RSST} or many other sources for the Pollaczek–Khinchine formula) imply that
\begin{align*}
\psi_\varepsilon^*(0)=\frac{\lambda \mathbb{E}X^*}{c}=1-\frac{\lambda\varepsilon \mathbb{P}(X>a)}{c}\leqslant\psi(0)\leqslant1.
\end{align*}
By letting $\varepsilon \to 0^+$ in the last inequalities, we get $\psi(0)=1$, or equivalently $\varphi(0)=0$. Then, $\psi(u)=1$ for all $u\geqslant0$ is implied by the same Pollaczek–Khinchine formula observing $\varphi_\varepsilon^*(u)$ being the multiple of $\varphi_\varepsilon^*(0)$. Indeed,
\begin{align*}
\varphi_\varepsilon^*(u)&=\left(1-\frac{\lambda \mathbb{E}X^*}{c}\right)\left(1+\sum_{n=1}^{\infty}\left(\frac{\lambda \mathbb{E}X^*}{c}\right)^nF_I^{*n}(u)\right)\\
&=\varphi_\varepsilon^*(0)\left(1+\sum_{n=1}^{\infty}\left(\psi_\varepsilon^*(0)\right)^nF_I^{*n}(u)\right),\,u\geqslant 0,
\end{align*}
where
$$
F_I(u)=\frac{1}{\mathbb{E}X^*}\int_{0}^{u}\mathbb{P}(X^*>x)\,{\rm d}x
$$
and $F_I^{*n}$ denotes the $n$-fold convolution of $F_I$.
Here
$$
\sum_{n=1}^{\infty}\left(\psi^*(0)\right)^nF_I^{*n}(u)\leqslant
\sum_{n=1}^{\infty}F_I^{*n}(u)=\sum_{n=1}^{\infty}\mathbb{P}(\eta_1+\ldots+\eta_n\leqslant u)<\infty,
$$
because of Lemma \ref{lem:concentration}, where the non-negative independent and identically distributed random variables $\eta_1,\,\eta_2,\,\ldots$ are described by the distribution function $F_I$.

\end{proof}

\section{E.S. Andersen's model}\label{sec:Andersen}

In this section we prove Theorem \ref{thm:Andersen}.
\begin{proof}[Proof of Theorem \ref{thm:Andersen}]
    Arguing the same as proving Theorem \ref{thm:classical} in Section \ref{sec:Classical}, we can define the random vector $(X^*,\,X)$, its independent copies $(X_1^*,\,X_1),\,(X_2^*,\,X_2),\,\ldots$ and show that $\psi_\varepsilon^*(u)\leqslant \psi(u)\leqslant1$ for all $u\geqslant 0$. Let $S^*_n=\sum_{i=1}^{n}(X^*_i-c\theta_i)$ and $S_n=\sum_{i=1}^{n}(X_i-c\theta_i)$ for all $n\in\mathbb{N}$. Then, see \cite[eq. (10)]{EMbrechts},
$$
\psi_\varepsilon^*(0)=1-\exp\left\{-\sum_{n=1}^{\infty}\frac{\mathbb{P}(S^*_n>0)}{n}\right\},
$$
because of the net profit condition  $\mathbb{E}X^*-c\mathbb{E}\theta=-\varepsilon \mathbb{P}(X>a)<0$.

It is known that, see \cite[Thm. 4.1]{Spitzer1}, 
$
\mathbb{E}(X^*-c\theta)<0
$
implies
$$
\sum_{n=1}^{\infty}\frac{\mathbb{P}(S^*_n>0)}{n}<\infty,
$$
while $\mathbb{E}(X-c\theta)=0$
implies
$$
\sum_{n=1}^{\infty}\frac{\mathbb{P}(S_n>0)}{n}=\infty.
$$
Therefore
\begin{align*}
\varphi(0)\leqslant \varphi^*_{\varepsilon}(0)\leqslant \exp\left\{-\sum_{n=1}^{N}\frac{\mathbb{P}(S^*_n>0)}{n}\right\}
\end{align*}
for any $N\in\mathbb{N}$. By letting $\varepsilon\to0^+$ in the last inequalities, we obtain
\begin{align*}
\varphi(0)\leqslant \exp\left\{-\sum_{n=1}^{N}\frac{\mathbb{P}(S_n>0)}{n}\right\}
\end{align*}
and consequently $\varphi(0)=0$ as $N$ can be arbitrarily large and the series 
$$
\sum\limits_{n=1}^{\infty}\frac{\mathbb{P}(S_n>0)}{n}
$$ 
diverges. The equality $\psi(u)=1$ for all $u\geqslant 0$ is implied by the fact that $\varphi^*_\varepsilon(u)$ is the multiple of $\varphi^*_\varepsilon(0)$. Indeed, by the Pollaczek–Khinchine formula (see \cite[eq. (10)]{EMbrechts})
\begin{align*}
\varphi_\varepsilon^*(u)&=e^{-A}\left(1+\sum_{n=1}^{\infty}\left(1-e^{-A}\right)^nH^{*n}(u)\right)\\
&=\varphi_\varepsilon^*(0)\left(1+\sum_{n=1}^{\infty}\left(\psi_\varepsilon^*(0)\right)^nH^{*n}(u)\right),\,u\geqslant0,
\end{align*}
where
\begin{align*}
&A=\sum_{n=1}^{\infty}\frac{\mathbb{P}(S^*_n>0)}{n},\\
&H(u)=\frac{F_{+}(u)}{F_{+}(\infty)},\\
&F_{+}(u)=\mathbb{P}(S^*_{N^+}\leqslant u)\\
&N^+=\inf\{n\geqslant1:S^*_n>0\},\\
&S^*_n=\sum_{i=1}^{n}\left(X^*_i-c\theta_i\right),
\end{align*}
and $H^{*n}$ denotes the $n$-fold convolution of $H$. Proof of the considered theorem follows according to the comments at the end of the proof of Theorem \ref{thm:classical}.
\end{proof}


\end{document}